\documentclass[reqno]{amsart}
\usepackage{epsfig}
\usepackage{color}
\usepackage{amssymb,amsmath,amsthm,amstext,amsfonts}
\usepackage{psfrag}
\usepackage{url}
\usepackage{epstopdf}
\usepackage{epic,eepic}
\usepackage{upgreek}
\usepackage{amssymb}
\usepackage{mathrsfs}

\usepackage[colorlinks=true, linkcolor=blue, urlcolor=red, citecolor=blue, hyperindex]{hyperref}

\makeatletter \@addtoreset{equation}{section} \makeatother

\renewcommand\thetable{\thesection.\@arabic\c@table}

\theoremstyle{plain}
\newtheorem{maintheorem}{Theorem}
\newtheorem{maincorollary}{Corollary}

\newtheorem{proposition}{Proposition}[section]
\newtheorem{lemma}{Lemma}[section]
\newtheorem{corollary}{Corollary}[section]
\newtheorem{definition}{Definition}[section]
\newtheorem{remark}{Remark}[section]

\newcommand{\vep}{\varepsilon}

\newcommand{\supp}{\operatorname{supp}}

\newcommand{\cM}{\mathcal{M}}

\newcommand{\cU}{\mathcal{U}}

\newcounter{main}

\title[Uniform hyperbolicity revisited: 
index of periodic points and equidimensional cycles]
{Uniform hyperbolicity revisited: 
index of periodic points and equidimensional cycles}

\author[M\'ario Bessa]{M\'ario Bessa}
\address{M\'ario Bessa, Departamento de Matem\'atica, Universidade da Beira Interior, Rua Marqu\^es d'\'Avila e Bolama,
  6201-001 Covilh\~a
Portugal.}
\email{bessa@ubi.pt}

\author[Jorge Rocha]{Jorge Rocha}
\address{Jorge Rocha, Departamento de Matem\'atica, Universidade do Porto, 
Rua do Campo Alegre, 687, 
4169-007 Porto, Portugal}
\email{jrocha@fc.up.pt}

\author[Paulo Varandas]{Paulo Varandas}
\address{Paulo Varandas, Departamento de Matem\'atica, Universidade Federal da Bahia\\
Av. Ademar de Barros s/n, 40170-110 Salvador, Brazil $\&$ CMUP, University of Porto -Portugal}
\email{paulo.varandas@ufba.br}

\begin{document}

\begin{abstract}
In this paper we revisit uniformly hyperbolic basic sets and the domination of Oseledets splittings at periodic points. 
We prove that periodic points with simple Lyapunov spectrum
are dense in non-trivial basic pieces of $C^r$-residual diffeomorphisms on three-dimensional manifolds ($r\ge 1$). 
In the case of the $C^1$-topology we can prove that either all periodic points of a hyperbolic basic piece for a diffeomorphism $f$ have simple spectrum $C^1$-robustly (in which case $f$ has a finest dominated splitting into one-dimensional sub-bundles and all Lyapunov exponent functions of $f$ are continuous in the weak$^*$-topology)
or it can be $C^1$-approximated by an equidimensional cycle associated to periodic points with robust different 
signatures. The later can be used as a mechanism to guarantee the coexistence of infinitely many periodic points
with different signatures.
\end{abstract}

\date{\today}
\maketitle

\section{Introduction}

A global view of dynamical systems has been one of the leading problems considered by the dynamical
systems community. Based on the pioneering works of Peixoto and Smale, a conjecture proposed by Palis
in the nineties has constituted a route guide for a global description of the space of dynamical systems.
This program, that roughly describes complement of uniform hyperbolicity as the space of diffeomorphisms
that are approximated by those exhibiting either homoclinic tangencies or heteroclinic cycles, has been 
completed  with much success in the $C^1$-topology, where perturbation tools like the Pugh closing lemma, Franks' lemma, Hayashi's connecting lemma or Ma\~n\'e ergodic closing lemma developed for the characterization of structural stability are available (see e.g.~\cite{Ma1, Hay, Palis} and references therein).

Although the uniform geometric structures of invariant manifolds present in hyperbolic diffeomorphisms are a basic ingredient in achieving several core results and are quite well established, some important questions
on the regularity of finer dynamical properties still remain to be answer.

Motivated by the analysis of the regularity of the Lyapunov exponents similarly to the proof of the stability conjecture, 
periodic orbits and their eigenvalues should play a key role. 
We say that a periodic point $p$ of period $\pi(p) \ge 1$ for a diffeomorphism $f$ has \emph{simple spectrum} if all the eigenvalues of $Df^{\pi(p)}(p)$ are real and distinct. Since all periodic points for an Axiom A diffeomorphism, in the same basic piece, have the same index and are homoclinically 
related, then we will  attribute them a signature, which consists of an ordered list of the dimensions of their finest dominated splittings. Moreover, we will say that two periodic points $p,q$ have different signatures if the Oseledets splittings for $Df^{\pi(p)}(p)$ and $Df^{\pi(q)}(q)$, which consist of the finest dominated splittings at the points, 
are distinct.
Our purpose in this paper is to revisit uniformly hyperbolic basic and explore the notion of domination of 
Oseledets splittings and characterize the continuity points of Lyapunov exponent functions in this context. 
These results lie ultimately in the analysis of the eigenvalues and the finest dominated splitting over the periodic points.

In the first part of the paper (see Theorem~\ref{teo1}) we prove that periodic points with simple Lyapunov spectrum
are dense in non-trivial basic pieces of $C^r$-residual diffeomorphisms on three-dimensional manifolds ($r\ge 1$). 
In particular, any diffeomorphism with a hyperbolic basic piece can be $C^r$-approximated by a diffeomorphism 
with a dense set of periodic points in the continuation of the hyperbolic set. 
The proof of this result explores the
hyperbolic structure of hyperbolic basic pieces, hence it does not rely on the classical $C^1$-perturbation lemmas. 

In the case of the $C^1$-topology, one can describe the mechanisms to deduce simple Lyapunov spectrum 
for all invariant measures assuming  the same property at periodic points, in the spirit of the previous works ~\cite{BGV,Cao,Ca,ST}. On the one hand we prove that, if all periodic points of 
a hyperbolic basic piece for a diffeomorphism $f$ have simple spectrum and the same property holds in a $C^1$
neighborhood of $f$ (c.f. Definition~\ref{def:simplestar}) then $f$ has a finest dominated splitting into
one-dimensional sub-bundles (see Theorem~\ref{thm:1FDS}). 
In this case all Lyapunov exponent functions of $f$ 
are continuous in the weak$^*$-topology. 
On the other hand, if periodic points miss to have simple spectrum robustly, then, by a $C^1$-arbitrarily small perturbation, one can create an equidimensional cycle associated to a pair of periodic points which 
robustly exhibit different signatures (see Theorem~\ref{thm:dichotomy}). 
Finally, we justify that these equidimensional cycles associated to periodic points with different signatures
play a similar role to the one of tangencies and heterodimensional cycles by its instability character. In fact, 
the existence of homoclinic tangencies or heterodimensional cycles is often associated to the so-called
Newhouse phenomenon of persistence of infinitely many sources or sinks 
(see e.g.~\cite{Ne} and references therein).  In this setting, the existence
of a equidimensional cycle associated to periodic points with different signatures can be used to generate (by perturbation) infinitely many periodic points with any of the signatures of the 
generating periodic orbits (see Theorem~\ref{thm:explo}). 
Using the existence of Markov partitions we deduce 
that similar results hold in the time-continuous setting (see Corollary~\ref{cor:flows}).

\section{Preliminaries and statement of the main results}

\subsection{Hyperbolic, Oseledets and finest dominated splittings}

Let $M$ be a compact Riemannian manifold $M$ and $f\in\text{Diff}^{\, 1}(M)$.
Let ${\text{Per}(f)}$ denote the set of periodic points for $f$ and $\Omega(f) \subset M$ denote the non-wandering set of $f$. Given an $f$-invariant compact set $\Lambda\subseteq M$ we say that $\Lambda$ is a \emph{uniformly hyperbolic set} if there is a $Df$-invariant splitting  $T_{\Lambda} M = E^s \oplus E^u$ and constants $C>0$ and $\lambda\in (0,1)$ so that 
$$
\| Df^n(x)\mid_{E^s_x} \| \le C \lambda^n 
	\quad\text{and}\quad
	\| (Df^n(x)\mid_{E^u_x})^{-1} \| \le C \lambda^n 
$$
for every $x\in \Lambda$ and $n\ge 1$. We can always change the metric in order to obtain $C=1$. We refer to $T_{\Lambda} M = E^s \oplus E^u$ as the \emph{hyperbolic splitting} associated to $f$.
We say that $\Lambda$ is an \emph{isolated set} if there exists an open neighborhood $U\supset \Lambda$ such that $\Lambda=\cap_{n\in\mathbb{N}}f^n(U)$. Finally, we say that $\Lambda$ is \emph{transitive} if there exists $x\in\Lambda$ such that the $f$-orbit of $x$ is dense in $\Lambda$. Along the paper unless stated otherwise we always assume that $\Lambda$ is a uniformly hyperbolic, isolated and transitive set.  For simplicity we call such a set a \emph{hyperbolic basic piece} or simply a \emph{basic piece}. It is well-know that hyperbolic sets admit analytic continuations, that is, there exists a $C^1$-neighborhood 
$\mathcal{U}$ of $f$ and an open neighborhood $U$ of $\Lambda$ so that $\Lambda_g:=\cap_{n\in\mathbb{N}}g^n(U)$ is a basic piece for $g$ and $f|_{\Lambda}$ is topologically conjugated to $g|_{\Lambda_g}$. Finally, we say that a hyperbolic basic piece is \emph{non-trivial} if it does not consist of a hyperbolic periodic point.

We say that a $C^1$-diffeomorphism $f$ is \emph{Axiom A} if $\overline{\text{Per}(f)}=\Omega(f)$ and $\Omega(f)$ is a uniformly hyperbolic set. We observe that the non-wandering set of an Axiom A diffeomorphism can be decomposed in a finite number of basic pieces. 
If $\mu$ is an $f$-invariant probability measure, then it follows from Oseledets' theorem~\cite{O} 
that for $\mu$-almost every $x$ there exists a decomposition (\emph{Oseledets splitting})
$T_x M= E^1_x \oplus E^2_x\oplus \dots \oplus E_{x}^{k(x)}$ and, for $1\le i\le k(x)$, 
there are well defined real numbers
$$
\lambda_i(f,x)= \lim_{n\to\pm \infty} \frac1n \log \|Df^n(x) v_i\|,
	\quad \forall v_i \in E^i_x\setminus \{\vec0\}
$$
called the \emph{Lyapunov exponents} associated to $f$ and $x$. It is well known that, if $\mu$ is 
ergodic, then the Lyapunov exponents are almost everywhere constant and $k(x)=k$ is constant. In this case the Lyapunov exponents are denoted simply by 
$\lambda_i(f,\mu)$. The \emph{Lyapunov spectrum} of a probability measure is the collection of all 
its Lyapunov exponents. If $\Lambda$ is a hyperbolic set for $f$ and $\mu$ is an $f$-invariant and ergodic probability measure supported on $\Lambda$, 
Poincar\'e recurrence theorem implies that $\supp(\mu) \subset \Omega(f)$ and, consequently, the decomposition 
$T_x M= E^1_x \oplus E^2_x\oplus \dots \oplus E_{x}^{k}$ is finer than $T_xM=E^s_x \oplus E^u_x$.
In particular $\mu$ is hyperbolic i.e. has only non-zero Lyapunov exponents.

\begin{definition}
Given a $Df$-invariant decomposition $F=E^1\oplus E^2 \subset TM$ over a basic piece $\Lambda$ of a diffeomorphism $f$ we say $E^1$ is dominated by $E^2$ if there exists $C>0$ and $\lambda \in (0,1)$ so that 
$$
\| Df^n(x) \mid_{E_x^1}\| . \| (Df^n (x)\mid_{E^2_x})^{-1}\| \le C \lambda^n
\text{ for every $n\ge 1$ and $x\in \Lambda$.}
$$
\end{definition}

It is known that we can change the metric in order to obtain $C=1$. It is clear from the definition that any hyperbolic splitting is necessarily a dominated splitting.

\begin{definition}\label{def:simple}
We say that $f$  has a \emph{one-dimensional finest dominated splitting} over a basic piece $\Lambda$ 
if there exists a continuous $Df$-invariant decomposition $T_x M= E^1_x \oplus E^2_x\oplus \dots \oplus E_{x}^{\dim M}$  in one-dimensional subspaces at every $x\in \Lambda$ such that $E^i_x$ is dominated by 
$E^{i+1}_x$  for every $i=1 \dots \dim M -1$.
\end{definition}

It is clear from the previous definition that having a one-dimensional finest dominated splitting is a
$C^1$-open condition for $C^r$-diffeomorphisms $(r\ge 1)$.
Since dominated splittings vary continuously with the base point on a basic set $\Lambda$, if the diffeomorphism $f$ has a one-dimensional finest dominated splitting
$T_{\Lambda} M= E^1 \oplus \dots \oplus E^{\dim M}$ then there exists a continuous function
$$
\Lambda \ni x \mapsto (v_x^1, v_x^2, \dots, v_x^{\dim M}) \in (T_x^1M)^{\dim M}
$$
such that $v_x^i\in E^i_x$ for every $1\le i \le \dim M$ and $x\in \Lambda$ and, consequently, the limit superior 
\begin{align}\label{eq:poitwise}
\overline\lambda_i(f,x) & := \limsup_{n\to\infty} \frac1n \log \|Df^n(x) v_i\| 
\end{align}
is everywhere well defined and it does not depend on the choice of the vector $v_i\in E^i_x \setminus\{0\}$.
Defining $S: T^1_\Lambda M \to T^1_\Lambda M$ and $\phi: T^1_\Lambda M \to \mathbb R$ by 
$$
S(x,v) =\left(f(x), \frac{Df(x) v}{\| Df(x) v\|}\right)
	\quad\text{and}\quad
	\phi(x,v) = \log \|Df(x) v\|
$$
then it follows that each value $\overline\lambda_i(f,x)$ can be written as
\begin{align}\label{eq:poitwise2}
\overline\lambda_i(f,x) 
	 = \limsup_{n\to\infty} \frac1n \sum_{j=0}^{n-1} \phi( S^j(x, v_i^x)).
\end{align}
Using ~\eqref{eq:poitwise2}, for any $f$-invariant and ergodic probability measure $\mu$ each of the values $\overline\lambda_i(f,x)$ coincide  $\mu$-almost everywhere coincide with the Lyapunov exponents $\lambda_i(f,\mu)$, for every 
$1\le i \le \dim M$. More generally, for any $f$-invariant probability measure $\mu$  we define their \emph{Lyapunov 
exponents} by 
\begin{equation}\label{eq:poitwise2}
\lambda_i(f,\mu):= \int \overline \lambda_i(f,x) \, d\mu(x)
\end{equation}
for every $1\le i \le \dim M$. 

\begin{remark}
If $\mu$ is an $f$-invariant probability measure, the ergodic decomposition theorem guarantees that $\mu=\int\mu_x\,d\mu(x)$ where each $\mu_x$ is an $f$-invariant and ergodic probability measure defined for
$\mu$-almost every $x\in\Lambda $, and $\frac1n\sum_{j=0}^{n-1} \delta_{f^j(x)}$ converges to $\mu_x$ in the weak$^*$-topology. 
Thus for $\mu$-almost every $x\in\Lambda$ and every $1\le i \le \dim M$, 
$$
\overline \lambda_i(f,x)=\lambda_i(f,\mu_x) = \int \phi(y,v_i^y) \, d\mu_x(y).
$$ 
Equality in \eqref{eq:poitwise2} can also be written as
\begin{equation}\label{eq:poitwise3}
\lambda_i(f,\mu)
	= \int \overline \lambda_i(f,x) \, d\mu(x) 
	= \int \lambda_i(f,\mu_x) \,d\mu(x)
	= \int \phi(x,v_i^x) \, d\mu(x)
\end{equation}
for every $1\le i \le \dim M$.
\end{remark}

\subsection{Statement of the main results}

\subsubsection{Plenty of periodic orbits with simple spectrum on hyperbolic basic sets of $C^r$-diffeomorphisms}

In this subsection we discuss the simplicity of the Lyapunov spectrum for $C^r$-diffeomorphisms
($r\ge 1$) restricted to hyperbolic basic pieces. 

\begin{maintheorem}\label{teo1}
Assume $\dim M=3$ and $r\ge 1$. Let $f$ be a $C^r$-diffeomorphism and $\Lambda$ be a non-trivial hyperbolic basic set for $f$ and let $\cU$ be a $C^r$-open neighborhood of $f$ so that the hyperbolic continuation $\Lambda_g$ is well defined
for all $g\in \cU$. Then, there exists a $C^r$-residual subset $\mathcal{R} \subset \cU$ such that if $g\in\mathcal{R}$ 
then the set of periodic orbits with simple spectrum is dense in $\Lambda_g$. 
\end{maintheorem}

Some comments are in order.
In the case that $r=1$ there are some perturbation tools sufficient to guarantee the previous result
(e.g. the methods of \cite{BGV} used in the proof of Theorem~\ref{thm:dichotomy}).
In the case that $r>1$, where there is a lack of perturbation tools, we use strongly both the robustness of the 
hyperbolic splitting and the persistence of the hyperbolic basic piece to prove that generic diffeomorphisms 
have a dense set of periodic orbits with simple spectrum on non-trivial basic pieces. 
We also observe that the three-dimensional assumption on the manifold is crucial. 

\subsubsection{The simple star property and one-dimensional finest dominated splittings}

In this subsection we provide a criterium over the periodic points for a hyperbolic basic piece to admit a one-dimensional finest dominated splitting.

\begin{definition}\label{def:simplestar}
Given a basic piece $\Lambda$ for a diffeomorphism $f$ we say that $f|_{\Lambda}$ is \emph{simple star} if there exists a $C^1$-open neighborhood 
$\mathcal U$ of $f$ so that, for any $g\in \mathcal U$ all periodic points for $g|_{\Lambda_g}$ have simple spectrum. 
\end{definition}

We observe that after \cite{Ma1}, a $C^1$-diffeomorphism is called star if it admits a $C^1$-open neighborhood formed by
diffeomorphisms whose periodic points are all hyperbolic. Since the non-wandering set is uniformly hyperbolic 
then it is clear that the set of Axiom A diffeomorphisms forms a $C^1$-open subset 
of star diffeomorphisms.
It follows from ~\cite{BGV} that periodic points are
enough to determine dominated splittings at the closure of periodic points. Since periodic points are dense in basic pieces we obtain the following result.

\begin{maintheorem}\label{thm:1FDS}
If $f|_{\Lambda}$ is  simple star, then $f$ admits a one-dimensional finest dominated splitting over $\Lambda$.
Moreover, the Lyapunov spectrum is simple for all $f$-invariant ergodic probability measures and
the map
\begin{equation}\label{LE}
\begin{array}{ccc}
\mathscr{L}:\mathcal M_1(f|_{\Lambda}) & \to & \mathbb R^{\dim M} \\
	\mu & \mapsto & (\lambda_{1}(f,\mu), \dots, \lambda_{\dim M}(f,\mu))
\end{array}
\end{equation}
is continuous, where $\mathcal M_1(f|_\Lambda)$ denotes the space of $f$-invariant probability measures supported in $\Lambda$ and endowed with the weak$^*$ topology. 
\end{maintheorem}

We observe that if the diffeomorphism $f$ has a one-dimensional finest dominated splitting over $\Lambda$,
then every $f$-invariant probability measure supported on $\Lambda$ admits $d=\dim M$ distinct Lyapunov exponents that differ by
a definite amount. Thus, the following is a direct consequence of Theorem~\ref{thm:1FDS}.

\begin{maincorollary}\label{tadinho}
Let $\Lambda$ be a basic piece of a diffeomorphism $f$. The following properties are equivalent:
\begin{itemize}
\item[(a)] $f|_\Lambda$ is simple star; 
\item[(b)] $f|_\Lambda$ admits a one-dimensional finest dominated splitting;
\item[(c)] there exists a constant $ c>0$ so that any $f$-invariant probability measure $\mu$ supported in $\Lambda$ 
has $\dim M$ distinct Lyapunov exponents and $\lambda_{i+1}(f,\mu)- \lambda_{i}(f,\mu)\ge c >0$ for 
every  $i=1,\dots, \dim M-1$.
\end{itemize}
\end{maincorollary}

The later corollary ultimately means that if the Lyapunov spectrum of periodic points simple is $C^1$-robust,
then there exists a uniform gap in the Lyapunov spectrum of periodic points.

\subsubsection{A dichotomy for basic pieces}\label{subsec:dichotomy}

In this subsection we obtain a dichotomy of one-dimensional finest dominated splitting \emph{versus} equidimensional
cycles with robust different signatures for basic pieces. 
Let us introduce these notions more precisely. 
Let $\Lambda$ be a basic set of a diffeomorphism $f$ and $p\in\Lambda$ be a hyperbolic periodic point of period $\pi(p)\ge 1$. Given
the finest $Df^{\pi(p)}(p)$-invariant dominated splitting $E^u_p=E^u_1 \oplus \dots \oplus E^u_k$ we define the 
\emph{unstable signature at $p$} (and denote it by $\text{sgn}^u(p)$) to be the $k$-uple $(\dim E^u_1, \dots, \dim E^u_k)$. Stable signatures at periodic points are defined analogously and denoted by $\text{sgn}^s$. Unstable (resp. stable) signatures
describe the existence of finer splittings in the unstable (resp. stable) bundle of periodic points.

\begin{figure}[h]\label{Fig4}
\begin{center}
  \includegraphics[width=12cm,height=2.7cm]{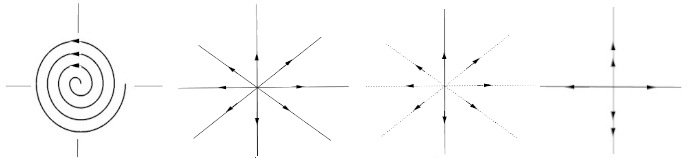}
\caption{\small{Different linear behavior of $Df^{\pi(p)}(p)\mid_{E^u_p}$ in the case of two-dimensional unstable subbundle corresponding to signatures $(2)$, $(2)$, $(2)$ and $(1,1)$, respectively. The third figure from the left corresponds to equal eigenvalues with a non-zero nilpotent part.}}
\end{center}
\end{figure}

\begin{figure}[h]\label{Fig3}
\begin{center}
  \includegraphics[width=10cm,height=2.3cm]{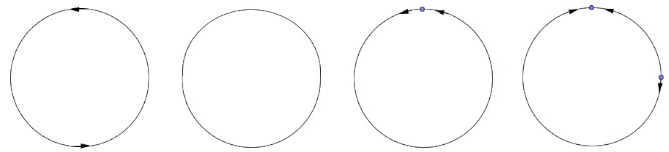}
\caption{Action of $Df^{\pi(p)}(p)\mid_{E^u_p}$ on the projective space in the case of two-dimensional unstable sub-bundle ordered as Figure 1.}
\end{center}
\end{figure}

\begin{definition}\label{def:equidim}
We say that a $C^1$ diffeomorphism $f$ exhibits an \emph{equidimensional cycle with robust different signatures} on a basic piece $\Lambda$ if there are two periodic points $p,q\in \Lambda$ with different stable or unstable signatures, and this property holds under arbitrarily $C^1$-small perturbations. 
\end{definition}

\begin{figure}[h]\label{fig1}
\begin{center}
  \includegraphics[width=8.5cm,height=4cm]{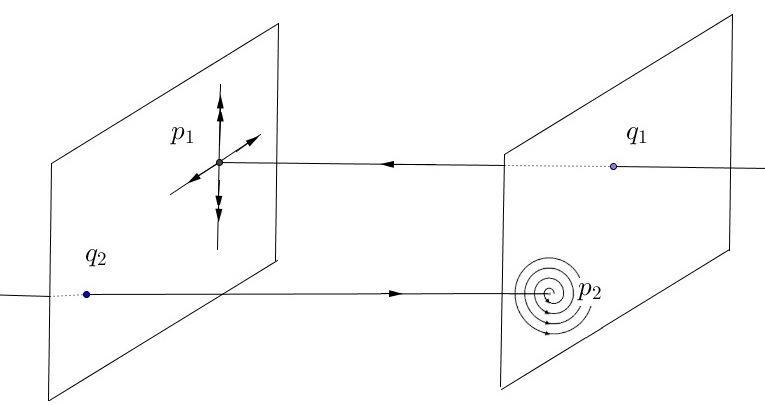}
\caption{Equidimensional cycle between two fixed points $p_1$ and $p_2$ such that $\text{sgn}^u(p_1)=(1,1)$ and $\text{sgn}^u(p_2)=(2)$.}
\end{center}
\end{figure}

In dimension three, the robustness of the different signatures is associated to the existence of a pair of periodic points whose action of the dynamical cocycle on the projective space of one periodic point exhibits a dominated splitting while for the other point it presents a rotational effect (c.f. Figures 1-3).
This is in strong analogy with the lack of domination observed in the presence of homoclinic tangencies.
In the space of $C^1$-diffeomorphisms the presence of homoclinic tangencies and heterodimensional cycles 
appears as a mechanism to create infinitely many sources or sinks (see e.g. the Newhouse phenomenon 
in \cite{Ne}). Roughly, if there exists a homoclinic tangency associated to a dissipative periodic point of 
saddle type then the lack of domination allows to perturb the system to mix the directions and to create 
a sink/source. The infinitely many sinks/sources arise associated to some phenomenon of persistence of tangencies.

\begin{maintheorem}\label{thm:dichotomy}
Assume that $\dim M =3$. Let $\Lambda$ be a non-trivial basic piece of a diffeomorphism $f$ and $\mathcal{U}$ a $C^1$-open neighborhood of $f$ on which the analytic continuation $\Lambda_g$ of $\Lambda$ is well-defined for all $g\in\mathcal{U}$. Either $f|_{\Lambda}$ has a one-dimensional finest dominated splitting, or $f$ can be $C^1$-approximated by diffeomorphisms $g$ that exhibit a 
equidimensional cycle in $\Lambda_g$ with robust different signatures. 
\end{maintheorem}

We notice that the property of having one-dimensional finest dominated splitting and the property of exhibiting a 
equidimensional cycle with robust different signatures are $C^1$-open. It is not hard to construct examples that lie in the boundary of the diffeomorphisms with a one-dimensional finest dominated splitting in a basic piece but not on the closure
of diffeomorphisms exhibiting a 
equidimensional cycle with robust different signatures, due to the existence of nilpotent part. In particular, there are examples of $C^1$ 
diffeomorphisms exhibiting periodic points (of low period) with equal eigenvalues that cannot be perturbed to create  
a complex eigenvalue. After this work was completed we were informed that, in the $C^1$-topology, Bochi and Bonatti
~\cite{BoBo} describe all Lyapunov spectra that can be obtained by perturbing the derivatives along periodic orbits of 
a diffeomorphism. 

The next result explains the conceptual similarity between equidimensional cycles with robust 
different signatures and the notion of heterodimensional cycles for dynamics beyond uniform hyperbolicity. 

\begin{maintheorem}\label{thm:explo}
Assume that $\dim M =3$ and let $\Lambda$ be a non-trivial basic piece of a diffeomorphism $f$. If $f|_{\Lambda}$ admits two periodic points $p,q$ with robust different unstable signatures, then there exists a 
$C^1$-open neighborhood $\mathcal U$ of $f$ and a residual subset $\mathcal R\subset \mathcal U$ such that for any $g\in \mathcal R$ we have that:
\begin{enumerate}
\item there exists a countable infinite set of points $P_1\subset Per(g)\cap \Lambda_g$ with $sgn^u(p_1)=sgn^u(p)$ for all $p_1\in P_1$ and
\item there exists a countable infinite set of points $P_2\subset Per(g)\cap \Lambda_g$ with $sgn^u(p_2)=sgn^u(q)$ for all $p_2\in P_2$,
\end{enumerate}
where $\Lambda_g$ stands for the hyperbolic continuation of the set $\Lambda$.
\end{maintheorem}

\subsubsection*{An application for flows}

In what follows we apply our results to the case of hyperbolic basic pieces for $C^1$-flows. By Bowen and  Ruelle~\cite{BR75}, the restriction of a $C^1$-flow to a hyperbolic basic piece for a $C^1$-flow admits a finite Markov partition and is semiconjugate to a suspension flows over a subshift of finite type and 
there exists a one-to-one correspondance between invariant measures of the discrete-time and 
continuous-time dynamical systems.
Indeed, given a measurable space $\Sigma$, a map $R\colon \Sigma\rightarrow{\Sigma}$, an $R$-invariant probability
measure $\tilde{\mu}$ defined in $\Sigma$ and a ceiling function
$h\colon \Sigma\rightarrow{\mathbb{R}^{+}}$ in $L^1(\tilde\mu)$ and 
bounded away from zero, 
consider the space $M_{h}\subseteq{\Sigma\times{\mathbb{R}_+}}$ defined by
$M_h=\{(x,t) \in \Sigma\times{\mathbb{R}_+}: 0 \leq t \leq h(x) \}/\sim$
where $\sim$ stands for the identification between the pairs $(x,h(x))$ and $(R(x),0)$. 
The semiflow defined on $M_h$ by $S^s(x,r)=(R^{n}(x),r+s-\sum_{i=0}^{n-1}h(R^{i}(x)))$, 
where $n\in{\mathbb{N}_0}$ is uniquely defined by 
$\sum_{i=0}^{n-1}h(R^{i}(x))\leq{r+s}<\sum_{i=0}^{n}h(R^{i}(x))$ 
is called a \emph{suspension semiflow}. If $R$ is invertible then $(S^t)_t$ is indeed a flow.
Since $h$ is bounded below, if $\text{Leb}_1$ denotes the one dimensional Lebesgue measure, then  $\eta \mapsto \bar\eta:=\frac{\eta \times \text{Leb}_1}{\int h \;d\eta}$
is a  one-to-one correspondence between $R$-invariant probability measures and $S^t$-invariant probability measures (we refer the reader \cite{BR75} for more details).

Hence, if $T_x \Sigma=E^1_x\oplus \dots E^k_x$
for a $\eta$-generic point $x$  then $T_{(x,s)} \Sigma=E^1_{(x,s)}\oplus \dots E^k_{(x,s)} \oplus \langle X \rangle$ 
is the Oseledets splitting associated to the vector field $X$ for $\bar\eta$-almost every $(x,s)$, where $\langle X \rangle$ is the subbundle generated by the vector field. In particular,  the Lyapunov spectrum of $(X,\bar\eta)$ is given by the 
Lyapunov exponents of $(S,\eta)$ and zero (corresponding to the flow direction).
Given the previous characterization, a hyperbolic basic piece $\Lambda$ for a vector field $X$, the vector field 
$X\mid_\Lambda$ is $C^1$-simple star if and only if the Poincar\'e map $S\mid_{\Sigma\cap \Lambda}$ is $C^1$-simple 
star. 
We also stress that such a Poincar\'e map inherits the information of the Lyapunov
exponents for both periodic points and all invariant measures: the Lyapunov exponents for the flow and 
for the Poincar\'e map differ from the integral of the return time function. The definitions from the 
discrete time setting extend to this setting in an analogous way. Since the $C^1$-perturbation tools used
in the proofs of Theorems~\ref{thm:1FDS} and ~\ref{thm:dichotomy} are available for flows, then the
following result is a consequence of the above mentioned theorems.

\begin{maincorollary}\label{cor:flows}
If $(X_t)_t$ is a simple star structurally stable flow, then $(X_t)_t$ admits a one-dimensional 
finest dominated splitting and the Lyapunov exponent functions vary continuously with the invariant
probability measure. 
Moreover, if $\dim M =4$ we have the following dichotomy: if $(X_t)_t$ is a structurally stable flow 
then either $(X_t)_t$ admits a one-dimensional finest dominated splitting,  or $(X_t)_t$ 
can be $C^1$-approximated by flows that exhibit a equidimensional cycle
associated points with robust different signatures. 
\end{maincorollary}

\section{Proofs}

In this section we prove our main results. 

\subsection{Proof of Theorem~\ref{teo1}}

The arguments here use strongly that $M$ is a three-dimensional manifold.
Let $r\ge 1$, let $f$ be a $C^r$-diffeomorphism and $\Lambda$ be a non-trivial hyperbolic basic set for $f$ and let 
$\cU$ be a $C^r$-open neighborhood of $f$ so that the hyperbolic continuation $\Lambda_g$ is well defined
for all $g\in \cU$. 

In particular, for every $g\in \mathcal U$ there exists
an homeomorphism $h_g: \Lambda \to  \Lambda_g$ close to the identity so that $h_g \circ f\mid_{\Lambda} = g\mid_{\Lambda_g} \circ h_g$.
Consequently, if $(p_j)_{j=1\dots \ell}$ are the periodic points of period $m$ for $f\mid_\Lambda$ then the continuations
$(h_g(p_j))_{j=1\dots \ell}$ are the periodic points of period $m$ for $g$.
Let $Per_s(g)$ denote the set of periodic points with simple Lyapunov spectrum.
The set 
$$
\mathcal U_n =\Big\{ g\in \mathcal U \colon \text{Per}_s(g) \, \text{is $\frac1n$-dense in}\, \Lambda_g \Big\}
$$
is clearly a $C^r$-open subset of $\cU$. We claim that $\cU_n$ is a $C^r$-dense subset of $\cU$.
This will be enough to prove the theorem because the residual subset
$\mathcal R:=\bigcap_{n\in \mathbb N} \mathcal U_n \subset \cU$ is so that the set of
periodic points with simple Lyapunov spectrum for any $g\in \mathcal R$
is dense in $\Lambda_g$. In other words, the residual subset $\mathcal R\subset \cU$ satisfies the
requirements of the theorem.
Thus, we are left to prove that for every $n\ge 1$ the set $\cU_n$ is a $C^r$-dense subset of $\cU$.

Before that we need a simple and useful lemma to perform the local $C^r$-per\-tur\-bations. This lemma is quite different from the well-known Franks' lemma because the approximation on the derivative depends on a pre-fixed size of neighborhood of the support of the perturbation.

\begin{lemma}\label{pert}
Given a $C^r$ diffeomorphism $f\colon M\rightarrow M$, $\epsilon>0$, $p\in M$ and $R>0$, there exists $\delta:=\delta(\epsilon,R)>0$ such that the following holds: if $A(p)\colon T_xM\rightarrow T_{f(x)}M$ is $\delta$-close to $Df(p)$ in the uniform norm of operators then there exists a $C^r$ diffeomorphism $g$ such that:
\begin{enumerate}
\item $g$ is $\epsilon$-$C^r$-close to $f$;
\item $g$ is supported in $B\left(p,R\right)$ and
\item $Dg(p)=A(p)$.
\end{enumerate}
\end{lemma}

\begin{proof}
For simplicity we assume all the computations in local charts in $\mathbb{R}^n$ and that $p$ is fixed. Take $\epsilon>0$ and $R>0$ and consider a $C^\infty$ \emph{bump function} $\varphi\colon [0,+\infty[\rightarrow [0,1[$ such that  $\varphi(t)=0$ if $t>\sqrt{R}$ and $\varphi(t)=1$ if $t< \frac{\sqrt{R}}{2}$. 
It follows from Fa\`{a} di Bruno's formula that, given $q\in B\left(p,R\right)$, we have
\begin{equation}\label{FDB}
\frac{\partial^n \varphi(\|q\|^2)}{\partial q_k}
	=\sum\frac{n!}{m_1!1!^{m_1}m_2!2!^{m_2}...m_n!n!^{m_n}}\varphi^{(m_1+...+m_n)}(\|q\|^2)
	\prod_{j=1}^n \left(\frac{\partial^j \|q\|^2}{\partial q_k}\right)^{m_j},
\end{equation}
where the sum is over all vectors with nonnegative integers entries $(m_1,..., m_n)$ such that we have 
$\sum_{j=1}^n j.m_j=n$. These derivatives are bounded above by a constant $C$ (depending on $R$).

Let $A\colon\mathbb{R}^n\rightarrow\mathbb{R}^n$ be a linear map $\delta$-close to $Df(p)$, where $\delta>0$ is very small (to be determined later on and depending on the derivatives in (\ref{FDB})).
Take a $C^r$ smooth map $\hat h\colon B\left(p,R\right)\rightarrow \mathbb{R}^n$  defined by $\hat h(q)=[Df(p)]^{-1}\cdot A(p)$, take $h(q)=\varphi (\|q\|^2)\hat h(q) + (1-\varphi (\|q\|^2)) Id$ and define $g:=f\circ h$. Clearly, by construction, $g$ coincides with $f$ in the complement of $B\left(p,R\right)$ and
$$Dg(p)=Df(h(p))\cdot Dh(p)=Df(p)\cdot D\hat h(p)=Df(p)\cdot [Df(p)]^{-1}\cdot A(p) =A(p).$$
This proves  (2) and (3). Furthermore, since $\hat h$ is constant and $\| \hat h(q)\|<\delta$, then the
${C^r}$-distance between $f$ and $g$ is proportional to $\|\hat h\|$ and can be taken smaller than $\epsilon$ provided
that $\delta$ is small enough. 
\end{proof}

We are now in a position to proceed with the proof that for every $n\ge 1$ the set $\cU_n$ is a $C^r$-dense subset of $\cU$. In fact this is a consequence of the following 

\begin{proposition}
Given $\epsilon>0$, $n\ge 1$ and a non-trivial basic piece $\Lambda$ 
of a $C^r$-diffeomorphism $f$, there exists a $C^r$-diffeomorphism $g$ that is $\epsilon$-$C^r$-close to $f$ and displaying a $g$-periodic orbit $q$ that is
$\frac1n$-dense in $\Lambda_g$ and has simple Lyapunov spectrum.
\end{proposition}

\begin{proof}
Let $\epsilon>0$, $f$ be a $C^r$ diffeomorphism, and $\Lambda$ be a non-trivial basic piece for $f$. Assume without loss of generality that
$T_{\Lambda} M = E^s \oplus E^u$ with $\dim E^s=1$ and $\dim E^u=2$ (otherwise just consider the diffeomorphism 
$f^{-1}$). In this three-dimensional setting, $p\in \text{Per}(f)$ has simple Lyapunov spectrum if and only if there are no
 complex eigenvalues for $Df^{\pi(p)}(p)\mid_{E^u_p}$ or $Df^{\pi(p)}(p)\mid_{E^u_p}$ has no equal real eigenvalues. Notice that  this last mentioned situation has empty $C^r$-interior. 

Fix $p\in \text{Per}(f)$. Up to consider $f^{\pi(p)}$ we just assume that $\pi(p)=1$ and that the restriction of $f$ to the 
non-trivial hyperbolic homoclinic class $\Lambda$ associated to $p$ is topologically mixing. In particular, $f\mid_{\Lambda}$ satisfies the specification property:
for each $\zeta>0$, there is an integer $N(\zeta)$ for which the following is true:
if $I_1,I_2,\cdots,I_k$
are pairwise disjoint intervals of integers with
$$\min\{|m-n|:m\in I_i,n\in I_j\}\ge N(\zeta)$$
for $i\not=j$ and
$x_1,\cdots,x_k \in X$ then there is a point $x\in X$ such that
$d(f^j(x),f^j(x_i))\le\zeta$ for $j\in I_i$ and $1\le i\le k$.
Every topologically mixing hyperbolic elementary set of a diffeomorphism satisfies
the specification property (see \cite{B}).

Consider a $\frac1{2n}$-dense set $F$ in $\Lambda$. Then there exists $\ell=\ell(n)\ge 1$ so that
for any large $m$ there exists a periodic point $x_m\in \Lambda$  of period equal to $m+\ell\times \# F$, such that the first 
$m$ iterates of $x_m$ belong to a small neighborhood $U_0$ of $p$ and the piece of orbit $\{f^{j+m}(x_m) : 1\le j \le \ell\}$ 
is $\frac1n$-dense.

 Observe that $\ell$ is fixed and depends only on $n$. Defined in this way, as $m$ increases, the point $x_m$ is sufficiently close  to $p$ and spends a large amount of time near $p$, and the matrix $Df^{m}(x_m)$ inherits the dynamical behavior of $[Df(p)]^m$ whereas the matrix $Df^{m+\ell}(x_m)$, corresponding to the periodic orbit computed along the whole orbit of $x_m$. 
 Define the compact $f$-invariant hyperbolic sets 
\begin{equation}\label{Knf}
K_n(f):=\overline{\bigcup_{m\geq n}\bigcup_{i\in\mathbb{Z}} f^{i}(x_m)} \subset \Lambda, 
\;\; n\ge 1.
\end{equation}
 The persistence of the hyperbolic splitting $E_{f}^u\oplus E_{f}^s$ for the map $f$ under $C^1$ (thus $C^{r}$) perturbations allows us to conclude that any $g$ sufficiently $C^{r}$-close to $f$ also has a hyperbolic splitting $E_{g}^u\oplus E_{g}^s$ (with a two-dimensional unstable and one-dimensional stable bundles $E_g^u$ and $E_g^s$, respectively) over the 
 $g$-invariant compact set $K_n(g)$ obtained as in \eqref{Knf} by taking the  continuation of the hyperbolic periodic 
 points $(x_m)_m$.

Now we will show that some $C^r$-small perturbation $g$ can be made in order to find a periodic point $x_m$ in $K_n(g)$
such that all the eigenvalues of $Dg^{\pi(x_m)}(x_m)$ are real and distinct, where $\pi(x_m)=m+\ell\times \# F$. Such point $x_m$ is
$\frac1n$-dense by construction.
Actually, although explicitly defined at a neighborhood of $p$, the detailed construction of the perturbation map necessarily makes use of bump-functions. For shortness we shall omit the full detailed construction. For any $n\ge 1$, let $\rho(m,f)$ stand for the rotation number of  $Df^{\pi(x_m)}(x_m)\mid_{E^u_{x_m}}$. Given a continuous arc of maps $(f_t)_{t\in I}$ near $f$ let $\delta(m,f_t)$ denote the oscillation of the rotation number along the arc, that is, 
$\delta(m,(f_t)_t)=\sup \{ | \rho(m,f_t) - \rho(m,f_s) | \colon s,t \in I \}$.
The next claim corresponds to \cite[Lemma 9.3]{BoV04} in our setting.

\bigskip

\noindent\textbf{Claim:}
There exists a continuous arc of maps $\{f_t\}_{t\in[0,1]}$ satisfying $f_0=f$ and $\|f_t-f_0\|_{C^r} <\epsilon$ for all $t\in[0,1]$ so that the following holds: for every $t$ there exists $m_t\in\mathbb{N}$ so that $\delta(m,(f_s)_{s\in[0,t]})>1$ for all $m\geq m_t$.

\bigskip

\noindent\emph{Proof of the claim:} We consider from now on a basis adapted to the splitting  $E^u_p\oplus E^s_p$ (for $f$), a constant $\sigma\in(0,1)$, $\gamma>1$ and maps $h_{t\xi}$ with support in a small ball $B(p,R)\subset U_0$ 
such that:
$$Df(p)=\begin{pmatrix}\sigma&0&0\\0&\gamma\cos\theta&-\gamma\sin\theta\\0&\gamma\sin\theta&\gamma\cos\theta\end{pmatrix}\,\,\,\,\,\text{and}\,\,\,\,\,Dh_{t\xi}(p)=\begin{pmatrix}1&0&0\\0&\cos (t\xi)&-\sin(t\xi)\\0&\sin(t\xi)&\cos(t\xi)\end{pmatrix}$$
As in Lemma~\ref{pert}, consider the $C^r$-diffeomorphisms $f_t=h_{t\xi}(p)\circ f(p)$ and let $\xi$ be sufficiently small in order to have $f_t$ $\epsilon$-$C^r$-close to $f$ for all $t\in[0,1]$. 
Since the uniformly hyperbolic splitting varies continuously with the point and the angle between stable and unstable bundles is bounded away from zero each fiber $E^u_{x}$ and $E^s_{x}$ is a graph over $E^u_{p}$ and $E^s_{p}$ (for all maps
in the family $(f_t)_{t\in [0,1]}$) for all points $x$ in the hyperbolic isolated set.
Identifying all the stable and unstable subspaces with $E^s_p$ and $E^u_p$ 
by parallel transport, consider the decomposition
$$
Df_t^{\pi(x_m)}(x_m)=\alpha_{t,m,m} . (\dots) . \,\alpha_{t,m,1} . \,\beta_{t,m},
$$
where the $\alpha$'s correspond to the derivative $Df_t$ at iterates of $x_m$ by $f_t$ inside the neighborhood $U_0$ 
of $p$ and the $\beta$'s to the derivative $Df_t$ at the iterates of $x_m$ by $f_t$ that are outside it. Since, for each $m$ 
we only allow $\ell(n)\times \#F$ iterates outside the neighborhood $U_0$ we obtain that, up to consider some subsequence if necessary, $\beta_{t,m}$ converges uniformly to 
some $\beta_t$ as $n\rightarrow\infty$ (i.e. the period of the $x_m$'s increases). 
We need to prove that for any $t$ there exists $m_t\in\mathbb{N}$ so that $\delta(m,(f_s)_{s\in[0,t]})>1$ for all $m\geq m_t$. Notice that each $\alpha_{t,m,i}\mid_{E^u}$ is uniformly close to the composition of an homothety by a rotation of angle $t\xi+\theta$ and so all give a contribution to the increase of the original rotation number by some positive constant and the claim is proved. 
\hfill $\square$

\medskip

By the claim, for any $\epsilon>0$ there exists a small $t$ and a large $m$ such that $f_t$ is $\epsilon$-$C^r$-close to $f$
and $\rho(m,f_t)\in\mathbb{Z}$. Therefore, $Df_t^{\pi(x_m)}(x_m)|_{E^u_{f_t}}$ has two equal real eigenvalues of norm larger than 1. Then, we perturb once more to obtain simple spectrum by simply consider a `diagonal' perturbation at $x_m$ 
given for $\eta_1,\eta_2> 0$ arbitrarily small and $\eta_1\not=\eta_2$, in a suitable base of eigenvalues, by the matrix
$$
\begin{pmatrix}1&0&0\\0&\eta_1&0\\0&0&\eta_2\end{pmatrix}.
$$
The argument to realize such perturbation of the derivative by a $C^r$-close diffeomoprhism follows the lines of Lemma~\ref{pert}. This completes the proof of the proposition.
\end{proof}

\subsection{Proof of Theorem~\ref{thm:1FDS}}

Assume that $f|_{\Lambda}$ is  simple star. We explore the fact that the set of (hyperbolic) periodic points is dense in 
$\Lambda$. If $\Lambda$ consists of a hyperbolic periodic point the result is immediate. So, we assume that $\Lambda$
is a non-trivial hyperbolic set.

Using that a dominated splitting extends to the closure of a set and that $f\mid_\Lambda$ is simple star (thus cannot be $C^1$-approximated by a diffeomorphism so that the Lyapunov spectrum at some periodic point is not simple) then, by Corollary~\ref{BGV}, there exist $m_1, n_1\ge 1$ so that there exists an $m_1$-dominated splitting 
$E\oplus F$ over the set $\Lambda\cap \overline{\bigcup_{\ell \ge n_1} Per_{\ell}(f)}$. 

If $E$ has dimension one we take $E^1=E$. Otherwise, repeating the procedure 
with $Df\mid_E$ there exists $n_2\ge 1$ and $m_2\ge 1$ (multiple of $m_1$) and a 
$m_2$-dominated splitting
$
	 T_{\Lambda \cap \overline{\bigcup_{\ell \ge n_2} Per_{\ell}(f)}} M 
	= (\hat E^1 \oplus \hat E^2) \oplus F
$
with $0<\dim \hat E^1, \dim \hat E^2 <\dim E$. If $\hat E^1$ has dimension one we take $E^1=\hat E^1$. Proceeding recursively with $E$ and $F$ it follows that there exist $m \ge 1$ 
and an $m$-dominated splitting 
$$
T_{\Lambda \cap \;\overline{\bigcup_{\ell \ge n} Per_{\ell}(f)}}\, M = E^1 \oplus \dots \oplus E^{\dim M}.
$$
Using that $\Lambda$ is a (non-trivial) homoclinic class and that every non-trivial homoclinic class contains 
infinitely many periodic points then $\bigcup_{\ell \ge n} Per_{\ell}(f)$ is dense in $\Lambda$.
Thus $f\mid_\Lambda$ has a one-dimensional finest 
dominated splitting, and the first claim in the theorem follows.

Now, we shall prove that the Lyapunov exponent function associated to $f\mid_\Lambda$ is continuous. In fact, 
let $T_{\Lambda} M= E^1 \oplus \dots \oplus E^{\dim M}$ be the one-dimensional, $Df$-invariant finest dominated 
splitting associated to $f$. 
Since dominated splittings varies continuously with the point
then there exists a continuous function
$$
\Lambda \ni x \mapsto (v_x^1, v_x^2, \dots, v_x^{\dim M}) \in (T_x^1M)^{\dim M}
$$
such that $v_x^i\in E^i_x$ and, by \eqref{eq:poitwise3}, 
the Lyapunov exponents are given by the integrals of continuous functions
$\lambda_i(f,\mu) = \int \phi(x,v_i^x) \, d\mu(x)$ for every $1\le i \le \dim M$ and $x\in \Lambda$.
 
It is immediate that these vary continuously with respect to $\mu$ in the weak$^*$ topology
and, consequently, the Lyapunov exponent map ~\eqref{LE} is continuous. 
By the dominated splitting property it follows that the Lyapunov spectrum of every invariant measure $\mu$ is simple. 
This finishes the proof of Theorem~\ref{thm:1FDS}.

\begin{remark}
We observe that an alternative argument for the simplicity of the Lyapunov spectrum for all invariant measures 
follows from the denseness of periodic measures in the set of $f$-invariant probability measures, by 
specification~\cite{Sigmund}.
\end{remark}

\begin{remark}
Since dominated splittings also vary continuously with the diffeomorphism,
then every $g$ that is $C^1$-close to $f$, then $\Lambda_g$ is a basic piece with a one-dimensional 
finest dominated splitting and there exists a continuous function
$$
(x,g) \mapsto (v_{x,g}^1, v_{x,g}^2, \dots, v_{x,g}^{\dim M}) \in (T_x^1M)^{\dim M}
$$
such that $v_x^i\in E^i_x$. Since Lyapunov exponents are computed as integrals of continuous
functions (cf. \eqref{eq:poitwise3}) then given a sequence $(f_n)_n$ convergent to $f$ in the $C^1$-topology,
and invariant measures $\mu_n  \in \cM_1(f_n)$ that converge to $\mu\in \cM(f)$ then we also obtain that 
$\lim\limits_{n\to\infty} \lambda_i(f_n,\mu_n)=\lambda_i(f,\mu) $
for every $1\le i \le \dim M$.
\end{remark}

\subsection{Proof of Theorem~\ref{thm:dichotomy}}\label{second}

Assume that $\dim M =3$,  $\Lambda$ is a non-trivial basic piece of a diffeomorphism $f$ and 
$\mathcal{U}$ a $C^1$-open neighborhood of $f$ on which the analytic continuation $\Lambda_g$ of $\Lambda$ is well-defined for all $g\in\mathcal{U}$.
To prove the theorem we are reduced to prove that if $f|_{\Lambda}$ cannot be $C^1$-approximated by diffeomophisms 
$g$ so that $g\mid_{\Lambda_g}$ has a one-dimensional finest dominated splitting then $f$ can be $C^1$-approximated 
by diffeomorphisms $g$ that exhibits a equidimensional cycle in $\Lambda_g$ with robust different signatures. 
Assume that $T_\Lambda M= E^s\oplus E^u$ with $\dim E^s=1$ and $\dim E^u=2$. We use the following direct consequence of \cite[Corollary~2.18]{BGV} for the cocycle $A=Df\mid_{E^u}$ together with Franks' lemma and, henceforth, is specific of the $C^1$-topology.

\begin{corollary}\label{BGV}\cite[Corollary~2.18]{BGV}
For any $\epsilon>0$ there are two integers $m$ and $n$ such that, for any periodic
point $x$ of period $\pi(x) \geq n$:
\begin{enumerate}
\item either $Df\mid_{E^u}$ admits an $m$-dominated splitting along the orbit of $x$ or else;
\item for any neighborhood $U$ of the orbit of $x$, there exists an $\epsilon$-perturbation $g$ of $f$
in the $C^1$-topology, coinciding with $f$ outside $U$ and on the orbit of $x$, and such that
the differential $Dg^{\pi(x)}(x)\mid_{E^u_x}$ has all eigenvalues real and with the same modulus.
\end{enumerate}
\end{corollary}

Since all periodic points are homoclinically related, we are reduced to prove first that 
there are (up to a $C^1$-perturbation) two periodic points with robust different signatures. 
We organize our argument in three cases.

\vspace{.2cm}
\noindent\emph{Case~1: $Df^{\pi(p)}(p)\mid_{E^1_p}$ has a complex eigenvalue for every $p\in\text{Per}(f)$}.
\vspace{.1cm}

If this is the case then $\text{sgn}^u(p)=(2,0)$ for every $p\in\text{Per}(f)$.
By Corollary~\ref{BGV}, item (2), there is a diffeomorphism $g$, $C^1$-arbitrarily close to $f$, and 
a periodic point $p$ for which $Dg^{\pi(p)}(p)\mid_{E^1}$ has two real eigenvalues with the same modulus. 
Since this is a non-generic property, by a $C^1$-small perturbation $\tilde g$ of $g$, the continuation 
$p_{\tilde g}$ of the periodic point $p$ is so that $D\tilde g^{\pi(p_{\tilde g})}(p_{\tilde g})\mid_{E^1}$ has real and simple spectrum. Since $\tilde g$ can be obtained displaying some periodic point $q$
in the basic piece such that $D\tilde g^{\pi(q)}(q)\mid_{E^1_q}$ has a complex eigenvalue (hence $\text{sgn}^u(q)=(2,0)$ robustly),
then $\tilde g \in \mathcal{SS}^1_c$.

\vspace{.2cm}
\noindent\emph{Case~2: 
There are $p,q\in Per(f)$ such that $Df^{\pi(p)}(p)\mid_{E^1_p}$ has a complex eigenvalue and 
 $Df^{\pi(q)}(q)\mid_{E^1_q}$ has two real eigenvalues}.
\vspace{.1cm}

If $Df^{\pi(q)}(q)\mid_{E^1_q}$ has simple real spectrum we are done. Otherwise, since there are periodic points whose derivative at the period has a complex eigenvalue
in this case a simple and arbitrary small $C^1$-perturbation $g$ supported in a neighborhood of $q$ in such a way that $Dg^{\pi(q)}(q)\mid_{E^1_q}$ has real simple 
spectrum and, consequently, an equidimensional cycle with robust different signatures.

\vspace{.2cm}
\noindent\emph{Case~3: for all $p\in\text{Per}(f)$ the linear map 
$Df^{\pi(p)}(p)\mid_{E^1_p}$ has only real eigenvalues}.
\vspace{.1cm}

In this case there are three situations to consider, depending on the number of periodic points with two real equal eigenvalues (presenting a nilpotent part or not). \vspace{.1cm}

(a) all periodic points $p\in\text{Per}(f)$ are so that 
$Df^{\pi(p)}(p)\mid_{E^1_p}$ has real simple spectrum (or, in other words, $\text{sgn}^u(p)=(1,1)$).
\vspace{.1cm}
As $f|_{\Lambda}$ cannot be $C^1$-approximated by diffeomophisms $g$ so that $g\mid_{\Lambda_g}$ has a one-dimensional finest dominated splitting, then it follows from Corollary~\ref{tadinho} that $f$ is not
a simple star diffeomorphism. Hence, it can be arbitrarily $C^1$-approximated by a diffeomorphism $g$
so that either there exists a periodic point $p$ such that $Dg^{\pi(p)}(p)\mid_{E^1}$
has one complex eigenvalue or two real equal eigenvalues. Such perturbation can be performed in order to maintain a periodic point with real simple eigenvalues. Moreover, the period of $p$ can be taken as larger as necessary. We are left to prove that a perturbation can be done in order to obtain complex eigenvalues (along $E^1$) for some periodic point in the case when we have two real equal eigenvalues for $Dg^{\pi(p)}(p)\mid_{E^1}$. We will deal with this situation in the next two cases (b) and (c).

(b) if there exists a periodic point $p$ so that $\dim \text{Ker}(Df^{\pi(p)}(p)\mid_{E^1_p}-I)=2$ then 
$Df^{\pi(p)}(p)\mid_{E^1_p}$ has two real eigenvalues larger than one with no nilpotent part. 
Hence, by a small $C^1$-perturbation whose support is contained in a neighborhood of $p$ 
one can obtain a $C^1$-diffeomorphism $g$ so that $Dg^{\pi(p)}(p)\mid_{E^1_p}$ has a complex
eigenvalue (this can be obtained via Franks' lemma and which can be written locally by $Dg=R_\theta 
\circ Df$, where $R_\theta$ represents the rotation of angle $\theta$ on the plane $E^1$ and the identity in $E^2$).

\vspace{.1cm}

(c) if all periodic points $p$ have two real equal eigenvalues and are such that 
 $\dim \text{Ker}(Df^{\pi(p)}(p)\mid_{E^1_p}-I)=1$ we shall borrow an idea from \cite[Section 9]{BoV04}
 and proceed as follows. In this case all  of these periodic points have a nilpotent part and, in a suitable 
 coordinate system, can be written
 by 
 $$
 Df^{\pi(p)}(p)\mid_{E^1_p}  
 	= \begin{pmatrix} \lambda_p & n(p) \\ 0 & \lambda_p \end{pmatrix}
 $$
where $n(p) \in \mathbb R$ denotes the nilpotent part and $\lambda_p >1$.
Actually, since $f$ is Axiom A then it admits a finite Markov partition $\mathcal P=\{P_i\}_i$. 
Let $n_{i,j}\ge 1$ be given such that $f^{n_{i,j}}(P_i)\cap P_j \neq \emptyset$.

Fix $p\in Per(f)$ as above and assume, for simplicity, that $p \in P_1$ is a fixed point. 
Given $n\ge 1$,  let $x_m \in \Lambda$ be a periodic point so that $f^j(x_m) \in P_1$ for every $1\le j \le n$,
$f^{n_{1,2}+n} (x_m) \in P_2$ and $x_m=f^{n_{2,1}+n_{1,2}+n} (x_m) \in P_1$.  For any $m\gg 1$ the 
$f$-invariant compact set $Y_m = \overline{ \bigcup_{n\ge m} \bigcup_{j\in \mathbb Z} f^j(x_m)}$
inherits the behavior of $Df(p)$. For periodic points 
$q \in Y_m$ with period $\pi(q)\ge 1$ one has 
 $$
 Df^{\pi(q)}(q)\mid_{E^1_q}  
 	\approx \begin{pmatrix} \lambda_p^{\pi(q)-\ell} & N(q) \\ 0 & \lambda_p^{\pi(q)-\ell} \end{pmatrix}\cdot A(q,\ell),
 $$
 where $A(q,\ell)$ is a $2\times2$ matrix corresponding to the $\ell=n_{2,1}+n_{1,2}$ iterates of transitions between $P_1$ and $P_2$,
and the nilpotent part $N(q)$ grows linearly with $\pi(q)$. In particular, $N(q)/  \lambda_p^{\pi(q)} \to 0$
as $\pi(q) \to\infty$. Given $\vep>0$ choose $q\in Y_m$ with large period so that $N(q) \le  \lambda_p^{\pi(q)-\ell} \vep$. 

First we perform a $\vep$-$C^1$-small perturbation along a piece of orbit of $q$ so that the resulting diffeomorphism
$g$ satisfies, in appropriate coordinates, 
 $$
 Dg^{\pi(q)}(q)\mid_{E^1_q}  
 	\approx \begin{pmatrix} \lambda_p^{\pi(q)-\ell} & 0 \\ 0 & \lambda_p^{\pi(q)-\ell} \end{pmatrix}\cdot A(q,\ell)=\lambda_p^{\pi(q)-\ell}  A(q,\ell).
	$$ 
If  $A(q,\ell)$ have complex eigenvalues we are done. 
 
Otherwise we notice that $A(q,\ell)=Df^{\ell}(f^{\pi(q)-\ell}(q))$ and that $q \mapsto A(q,\ell)$ is continuous. There exists an appropriate base, given by the Jordan canonical form, so that the matrix $A(q,\ell)$ can be 
written of the form 
$$
A_1:= \begin{pmatrix} a & 0 \\ 0 & b \end{pmatrix}
	\quad\text{or}\quad
A_2:=\begin{pmatrix} a & b \\ 0  & a \end{pmatrix}
$$
for some real numbers $a,b$ (depending on the $\ell$ iterates $$\{ f^{j} (f^{\pi(q)-\ell} (q)) \colon j=0 \dots \ell-1 \}$$ of the orbit of the periodic point $q$). 
By the continuity of the matrix $A(q,\ell)$ with the periodic point $q$ the real numbers $a,b$ can be taken bounded from above and below by uniform
constants independent on the periodic point $q$ and period $\pi(q)$.
The strategy is to perturb the derivative of the diffeomorphism $g$ along a piece of orbit $\{ f^{j}  (q) \colon j=0 \dots N\}$ ($N\ge 1$ to be determined below)
in such a way that the resulting 
$C^1$-diffeomorphism $\tilde g$ is so that  $D\tilde g^{\pi(q)}(q)\mid_{E^1_q}$ has complex eigenvalues. 

Given $\vep>0$ take $N(\vep)\approx \frac1\vep$ (more precisely, $N(\vep)\ge \frac{1}{1+\vep} \log \left|\frac{b}{a}\right|$ in case of matrix 
$A_1$ and $N(\vep)\ge \frac{b}{\vep}$ in case of matrix $A_2$) and take $q \in Y_m$ with period $\pi(q)\ge N(\vep)+\ell$. 
By Franks' lemma, in the case of a matrix of the form $A_1$ then we perturb the $C^1$-diffeomorphism over the piece of orbit $\{ f^{j}  (q) \colon j=0 \dots j=0 \dots N(\vep)-1 \}$ 
by concatenations of perturbations of the derivative to get a $C^1$-diffeomorphism $g$ whose derivative is obtained from the one of $f$ by perturbations ($C^0$-close to the identity) of the form
$$
\begin{pmatrix} 1+\vep & 0 \\ 0 & 1 \end{pmatrix}
$$
in an appropriate base. For the resulting diffeomorphism we get that 
$$
 Dg^{\pi(q)}(q)\mid_{E^1_q} 
 	\approx \lambda_p^{\pi(q)-\ell} \begin{pmatrix} (1+\vep)^{N(\vep)} a & 0 \\ 0 & b \end{pmatrix}
	\approx \lambda_p^{\pi(q)-\ell} \begin{pmatrix} b & 0 \\ 0 & b \end{pmatrix}.
 $$
In the case of a matrix of the form $A_2$ we proceed analogously perturbing the 
derivative of the diffeomorphism $f$ over the piece of orbit $\{ f^{j}  (q) \colon j=0 \dots N(\vep)-1 \}$ by perturbations of the form
$$
\begin{pmatrix} 1 & -\frac{\vep}{a} \\ 0 & 1 \end{pmatrix}
$$
in an appropriate base. Thus we get
$$
 Dg^{\pi(q)}(q)\mid_{E^1_q} 
 	\approx \lambda_p^{\pi(q)-\ell} \begin{pmatrix}  a & b-\vep N(\vep) \\ 0 & a \end{pmatrix}
	\approx \lambda_p^{\pi(q)-\ell} \begin{pmatrix} a & 0 \\ 0 & a \end{pmatrix}.
 $$
Hence there exists a periodic point $\tilde q$ such that $D\tilde g^{\pi(\tilde q)}(\tilde q)\mid_{E^1_{\tilde q}}$ 
has simple real spectrum. This finishes the proof of the theorem.

\subsection{Proof of Theorem~\ref{thm:explo}}

Let $f$ be a $C^1$ diffeomorphism and $\Lambda$ be a hyperbolic basic piece. Assume there are two periodic points 
$p,q \in \Lambda$ with robust different unstable signatures. Let $\mathcal{U}$ be a $C^1$-open neighborhood of $f$ so that for any $g\in\mathcal{U}$ the analytic continuation $\Lambda_g$ of $\Lambda$ is well defined and let $p_g$ and $q_g$ denote the continuations of $p$ and $q$, respectively, and $sgn^u(p_g)=sgn^u(p)$ and $sgn^u(q_g)=sgn^u(q)$. 
Since $\dim M =3$ we may assume without loss of generality that $\dim(E^s)=1$ and $\dim(E^u)=2$.

For any $n\in\mathbb{N}$ we define $\mathcal{R}_n$ to be the set of $C^1$ diffeomorphisms $g\in \mathcal{U}$ such that $g$ has $n$ distinct periodic points with robust unstable signature equal to the one of $p_g$ and has $n$ distinct periodic points with robust unstable signature equal to the one of $q_g$.

Observe that, by hypothesis, $\mathcal{R}_1=\mathcal{U}$ and so $\mathcal{R}_1$ is $C^1$-open and dense in $\mathcal{U}$. For any $n\geq 1$ the set $\mathcal{R}_n$ is clearly $C^1$-open. If one shows that each $\mathcal{R}_n$ is $C^1$-dense then the residual set $\mathcal{R}=\cap_n \mathcal{R}_n$ satisfies the statement of Theorem~\ref{thm:explo}. 

Assume that $\mathcal{R}_k$ is $C^1$-dense for any $k=1,...,n$ and fix $\epsilon>0$ and any $g\in\mathcal{R}_n$. We claim that there exists $g_1\in \mathcal{R}_{n+1}$ which is $\epsilon$-$C^1$-close to $g$. The diffeomorphism $g_1$ will be obtained from $g$ by $C^1$-small perturbations at two periodic points in order to obtain one more point of each robust signature. By Corollary~\ref{BGV} we know that there exists $\ell,m\in\mathbb{N}$ such that, for any periodic point $x$ of period $\pi(x) \geq \ell$ either there exists an $m$-dominated splitting along the orbit of $x$ or else for any neighborhood $U$ of the orbit of $x$, there exists an $\epsilon/4$-perturbation $g_1$ of $g$
in the $C^1$-topology, coinciding with $g$ outside $U$ and on the orbit of $x$, and for which the tangent map $(Dg_1)^{\pi(x)}_x|_{E^u_x}$ has a real eigenvalue with multiplicity two.

Assume without loss of generality that $\text{sgn}^u(p_g)=(1,1)$ and $\text{sgn}^u(q_g)=(2)$.
The existence of $q_g$ implies that the set of periodic points $x$ of period larger that $\ell$ and with an $m$-dominated splitting along the unstable fiber is not dense in $\Lambda_g$. Indeed, if this was not the case, then $q_g$ would have an $m$-dominated splitting on $E^u_{q_g}$. By the previous dichotomy it follows that $q_g$ is accumulated by open sets without periodic points displaying an $m$-dominated splitting on $E^u$. Thus we can pick a periodic point $x$ distinct from the $2n$ marked periodic point for $g$ in one of these open sets and with arbitrarily large period $\pi(x)$. Using the dichotomy there exists an $\epsilon/4$-perturbation $g_1$ of $g$
in the $C^1$-topology, coinciding with $g$ outside $U$ and on the orbit of $x$, and for which the tangent map $(Dg_1)^{\pi(x)}_x|_{E^u_x}$ has a real eigenvalue with multiplicity two. Finally, since the periodic point $x$ can be chosen with an arbitrarily large period we can proceed as in Subsection \ref{second} and perform an $\epsilon/4$-perturbation $g_2$ of $g_1$ in the $C^1$-topology so that $(Dg_2)^{\pi(x)}_x|_{E^u_x}$ has a complex eigenvalue. Clearly, $x$ has an unstable robust signature $\text{sgn}^u(x)=(2)$.

If $g_2$ already have $n+1$ distinct periodic points with robust unstable signature equal to the unstable signature of $p_{g}$ we are done. Otherwise, we are left to show that $g_2$ can be $\epsilon/4$-approximated in the $C^1$-topology by $g_3$ exhibiting one more distinct periodic point with robust unstable signature equal to the unstable signature of $p_{g}$. Indeed, if all but $n$ periodic points have unstable signature $(2)$ and since there is no dominated splitting restricted to the unstable fiber and along these orbits we can perform an $\epsilon/4$-approximated in the $C^1$-topology in order to obtain a distinct periodic point $x$ with robust unstable signature equal to $\text{sgn}^u(x)=(1,1)$. This completes the proof of the $C^1$-denseness of 
$\mathcal R_n$ in $\mathcal U$ and finishes the proof of the theorem.

\begin{remark}
The previous construction of infinitely many periodic points of a certain signature has the same flavor of the construction of a generic set with infinite sinks and sources for diffeomorphisms with homoclinic tangencies: the perturbation are localized in regions where there is a lack of domination among the stable and unstable bundles.
\end{remark}

\vspace{.4cm}
\textbf{Acknowledgements:} This work was partially supported by  CMUP (UID/MAT/ 00144/2013), which is funded by FCT (Portugal) with national (MEC) and European structural funds through the programs FEDER, under the partnership agreement PT2020. MB was partially supported by National Funds through FCT  Funda\c{c}\~ao para a Ci\^encia e a Tecnologia, project PEst-OE/MAT/UI0212/2011. PV was partially supported by a CNPq-Brazil. 


\end{document}